\documentclass[12pt,a4paper]{amsart}
\usepackage{amssymb,amsthm,mathrsfs,upgreek}
\usepackage{fullpage}
\usepackage{euscript}

\usepackage[arrow, matrix, curve]{xy}

\usepackage{hyperref}

\usepackage{graphicx}
\usepackage[normalem]{ulem}

\usepackage{tikz}
\usetikzlibrary{calc,intersections}

\usepackage{extarrows, mathtools}
\makeatletter
\newcommand*{\da@rightarrow}{\mathchar"0\hexnumber@\symAMSa 4B }
\newcommand*{\da@leftarrow}{\mathchar"0\hexnumber@\symAMSa 4C }
\newcommand*{\xdashrightarrow}[2][]{%
\mathrel{%
\mathpalette{\da@xarrow{#1}{#2}{}\da@rightarrow{\,}{}}{}%
}%
}
\newcommand*{\da@xarrow}[7]{%
\sbox0{$\ifx#7\scriptstyle\scriptscriptstyle\else\scriptstyle\fi#5#1#6\m@th$}%
\sbox2{$\ifx#7\scriptstyle\scriptscriptstyle\else\scriptstyle\fi#5#2#6\m@th$}%
\sbox4{$#7\dabar@\m@th$}%
\dimen@=\wd0 %
\ifdim\wd2 >\dimen@
\dimen@=\wd2 %
\fi
\count@=2 %
\def\da@bars{\dabar@\dabar@}%
\@whiledim\count@\wd4<\dimen@\do{%
\advance\count@\@ne
\expandafter\def\expandafter\da@bars\expandafter{%
\da@bars
\dabar@ 
}%
}%
\mathrel{#3}%
\mathrel{%
\mathop{\da@bars}\limits
\ifx\\#1\\%
\else
_{\copy0}%
\fi
\ifx\\#2\\%
\else
^{\copy2}%
\fi
}%
\mathrel{#4}%
}
\makeatother

\newcommand{\pr}{ \operatorname{{pr}}}

\newcommand{\Supp}{\operatorname{{Supp}}}
\newcommand{\red}{\operatorname{{red}}}
\newcommand{\reg}{\operatorname{{reg}}}
\newcommand{\Sing}{\operatorname{{Sing}}}

\newcommand{\Aut}{\operatorname{{Aut}}}
\newcommand{\SAut}{\operatorname{{SAut}}}

\newcommand{\GL}{\operatorname{GL}}

\newcommand{\Pic}{\operatorname{{Pic}}}

\newcommand{\Fl}{\operatorname{{Fl}}}
\newcommand{\W}{\operatorname{W}}

\newcommand{\aaa}{\mathrm{a}}
\newcommand{\s}{\mathrm{s}}
\newcommand{\m}{\mathrm{m}}

\newcommand{\g}{\mathrm{g}}

\newcommand{\Ga}{{\mathbb G}_{\mathrm{a}}}
\newcommand{\Gm}{{\mathbb G}_{\m}}
\newcommand{\ZZ}{{\mathbb{Z}}}
\newcommand{\QQ}{{\mathbb{Q}}}
\newcommand{\PP}{{\mathbb{P}}}
\newcommand{\FF}{{\mathbb{F}}}
\newcommand{\A}{{\mathbb{A}}}

\newcommand{\CC}{{\mathbb{C}}}

\newcommand{\fL}{{\mathfrak{L}}}

\newcommand{\SSS}{{\mathscr{S}}}
\newcommand{\NNN}{{\mathscr{N}}}
\newcommand{\OOO}{{\mathscr{O}}}

\newcommand{\xref}[1]{\textup{\ref{#1}}}
\renewcommand\labelenumi{\rm (\roman{enumi})}
\renewcommand\theenumi{\rm (\roman{enumi})}

\makeatletter
\@addtoreset{equation}{subsection}
\makeatother

\newtheorem{mthm}{Theorem}

\swapnumbers
\newtheorem{thm}[equation]{Theorem}
\newtheorem{cor}[equation]{Corollary}
\newtheorem{lem}[equation]{Lemma}
\newtheorem{claim}[equation]{Claim}
\newtheorem{prop}[equation]{Proposition}
\theoremstyle{definition}
\newtheorem*{definition*}{Definition}
\newtheorem{defi}[equation]{Definition}

\newtheorem{nota}[equation]{Notation}
\newtheorem{rem}[equation]{Remark}

\newtheorem{sit}[equation]{}

\title[Flexible cones over Fano-Mukai fourfolds]{Affine cones over Fano-Mukai fourfolds of genus $10$ are flexible}
\author{Yuri Prokhorov}\thanks{
The research of the first author
was partially supported by the HSE University Basic Research Program, Russian Academic Excellence Project ``5-100''.
}

\author{Mikhail Zaidenberg}
\address{\emph{Yuri Prokhorov}\newline
Steklov Mathematical Institute, Moscow, Russian Federation
\newline
National Research University Higher School of Economics, Moscow, 
Russian Federation
\newline
Department of Algebra, 
Moscow State Lomonosov University, Russian Federation
}
\email{prokhoro@mi-ras.ru}
\address{\emph{Mikhail Zaidenberg}\newline Universit\'e Grenoble Alpes, CNRS, Institut Fourier, F-38000 Grenoble, France}
\email{Mikhail.Zaidenberg@univ-grenoble-alpes.fr}

\keywords{Flexible affine variety, affine cone, Fano fourfold}
\subjclass {Primary 14J35, 14J45; Secondary 14R10, 14R20}
\begin{document}

\begin{abstract} 
We show that the affine cones over any Fano--Mukai fourfold of genus $10$ are flexible in the sense of \cite{AFKKZ13}. 
In particular, the automorphism group of such a cone acts highly transitively outside the vertex. 
Furthermore, any Fano--Mukai fourfold of genus $10$, with one exception, admits a covering by open charts isomorphic to $\A^4$. 
\end{abstract}
\date{} 
\maketitle

\setcounter{tocdepth}{1}\tableofcontents

\section*{Introduction} 
Our base field in this paper is the complex number field $\CC$. 
Let $X$ be an affine variety over $\CC$. Consider the subgroup $\SAut(X)$ of the automorphism group $\Aut(X)$ generated by all the one-parameter unipotent subgroups of $\Aut(X)$. The variety $X$ is called \emph{flexible} if ${\SAut}(X)$ acts highly transitively on the smooth locus ${\reg}(X)$, that is, $m$-transitively for any natural number $m$ \cite{AFKKZ13}. 
There are many examples of flexible affine varieties, see, e.g., \cite{AKZ12, AFKKZ13, MPS18}; 
studies on this subject are in an active phase. 
For a projective variety, the flexibility of the affine cones might depend on the choice of an ample polarization. 
If the affine cone over a smooth projective variety $V$ with Picard number one is flexible, then $V$ is a Fano variety. 
For a pluri-anticanonical polarization of a Fano variety $V$, the flexibility of the affine cone $X$ over $V$ is a nontrivial new invariant of $V$.
Among del Pezzo surfaces with their pluri-anticanonical polarizations, only the surfaces of degree $\ge 4$ have flexible affine cones, see \cite{Pe13, PW16}.
Moreover, the group $\SAut(X)$ of the affine cone $X$ over $(V, -mK_V)$, $m>0$, is trivial for any del Pezzo surface $V$ of degree at most $3$, see \cite{CPW16a, KPZ11}. 
The affine cones over flag varieties of dimension $\ge 2$ are flexible \cite{AKZ12}. 
The secant varieties of the Segre-Veronese varieties provide another class of examples \cite{MPS18}. 

There are many examples of Fano threefolds 
with flexible affine cones, see, e.g., the survey article~\cite{CPPZ20} and the references therein. However, we know just two examples of such Fano fourfolds, 
and one of these is a special Fano-Mukai fourfold of genus 10, see \cite[Theorem 14.3]{PZ18}. In the following theorem we extend the latter result to all the Fano-Mukai fourfolds of genus 10.

\begin{mthm}
\label{mthm} 
Let $V=V_{18}$ be a Fano-Mukai fourfold of genus $10$. 
Then the affine cone over $V$ is flexible for any ample polarization of $V$. 
\end{mthm}

The proof exploits the criteria of flexibility of affine cones borrowed from \cite[Theorem~5]{Pe13} and \cite[Theorem~1.4]{MPS18}; cf.~also \cite[Theorem~2.4]{Per20}. These criteria are based on existence of some special open coverings of the underlined projective variety, for instance, a covering by flexible affine charts, or by suitable toric affine varieties, or by the affine spaces, see Theorem~\ref{thm:criterion-MPS}. To apply these criteria, one needs to construct such a covering. To this end, we use Theorem~\ref{thm:main} below.

We use the notation of \cite{PZ18}. Let $V_{18}$ be a Fano-Mukai fourfold of genus $10$ and degree $18$ half-anticanonically embedded in $\PP^{12}$.
Recall (\cite{KapustkaRanestad2013}, \cite[Remark~13.4]{PZ18}) that the moduli space of these fourfolds is one-dimensional. 
It contains two special members, namely, $V^{\aaa}_{18}$ with $\Aut^0(V^{\aaa}_{18})=\Ga\times\Gm$ and $V^{\s}_{18}$ with 
$\Aut^0(V)=\GL_2(\CC)$ \cite[Theorem~1.3.a,b]{PZ18}. For the general member $V^{\g}_{18}\notin \{V^{\aaa}_{18},\, V^{\s}_{18}\}$ one has $\Aut^0(V^{\g}_{18})=\Gm^2$ \cite[Theorem~1.3.c]{PZ18}. 

Any Fano-Mukai fourfold $V=V_{18}$ can be represented, at least in two ways, as an $\Aut^0(V)$-equivariant compactification of the affine space $\A^4$ \cite[Theorem~1.1]{PZ18}. More precisely, there exists at least two different $\Aut^0(V)$-invariant hyperplane sections $A_1, A_2$ of $V$ such that $U_i:=V\setminus A_i$ is isomorphic to $\A^4$, $i=1,2$. Statement (a) of the following theorem is proven in \cite[Theorem~13.5(f)]{PZ18}. 

\begin{mthm}
\label{thm:main} For the Fano-Mukai fourfolds $V=V_{18}$ of genus $10$ the following hold.
\begin{enumerate}
\renewcommand\labelenumi{\rm (\alph{enumi})}
\renewcommand\theenumi{\rm (\alph{enumi})}
\item
\label{thm:main-a}
If $\Aut^0(V)=\GL_2(\CC)$ then $V$ admits a covering by a one-parameter family of Zariski open subsets $U_t$ isomorphic to $\A^4$, where each $U_t$ is the complement of a hyperplane section of $V$. Exactly two of these sets are $\Aut^0(V)$-invariant;
\item
\label{thm:main-b}
if $\Aut^0(V)=\Gm^2$ then $V$ is covered by six $\Aut^0(V)$-invariant subsets $U_i\cong\A^4$ as in~\ref{thm:main-a}, $i=1,\ldots,6$;
\item
\label{thm:main-c} if $\Aut^0(V)=\Ga\times\Gm$ then there exist in $V$ four $\Aut^0(V)$-invariant affine charts $U_i\cong\A^4$, $i=1,\ldots,4$ such that $V\setminus\bigcup_{i=1}^4 U_i$ is a projective line covered by a one-parameter family of $\A^2$-cylinders $U_t\cong \A^2\times Z_t$, where $Z_t$ is a smooth affine surface, $t\in \PP^1$. 
\end{enumerate}
\end{mthm}

Some other families of Fano varieties demonstrate similar properties:
the Fano threefolds of degree~$22$ and Picard number~$1$ \cite{Kuznetsov-Prokhorov, Prokhorov-1990c} and the Fano threefolds of degree~$28$ and
Picard number~$2$ \cite[\S~9]{Prz-Ch-Shr:19}. Any member of the first family 
is a compactification of the affine $3$-space. It is plausible that the same is true for the second family. 
It would be interesting to investigate the flexibility of the affine cones over these varieties.

The paper is organized as follows. In Sections~\ref{sec:prelim} and~\ref{sec:Aut-action} we gather necessary preliminaries, in particular, some results from \cite{KapustkaRanestad2013, PZ16, PZ18}; cf.~also \cite{PZ17}. Based on this, in Section~\ref{sec:Aut-action} we prove Theorems~\ref{mthm} and~\ref{thm:main} in the case where $\Aut^0(V)=\Gm^2$. In
Section~\ref{sec:general-case} we proceed with the proof of Theorem~\ref{thm:main} in the case $\Aut^0(V)=\Ga\times\Gm$. Besides, we provide a criterion of flexibility related to the existence of a special family of $\A^2$-cylinders on $V$. Such a family of cylinders is constructed in Section~\ref{section:quadric} for the smooth quadric fourfold $Q^4$ and for the del Pezzo quintic fourfold $W_5$, and in Section~\ref{sec:V18-aff-cones} for any Fano-Mukai fourfold $V$ of genus $10$. This enables us to complete the proofs of Theorems~\ref{mthm} and~\ref{thm:main} in the remaining case where $\Aut^0(V)=\Ga\times\Gm$. 
\subsubsection*{Acknowledgment.} 
The authors are grateful to Alexander Perepechko for a careful reading of the paper
and valuable remarks.

\section{Cubic scrolls in the Fano-Mukai fourfolds $V_{18}$}
\label{sec:prelim}
In this section we recall and extend some facts from \cite{PZ18} used in the sequel. 
Throughout the paper we let $V$ be a 
Fano-Mukai fourfold $V=V_{18}$ of genus $10$ half-anticanonically embedded in $ \PP^{12}$.
Following~\cite{KapustkaRanestad2013} we call a \emph{cubic scroll} both a smooth cubic surface
scroll and a cone over a rational twisted cubic curve. The latter cones in 
$\PP^4$ will be called \emph{cubic cones} for short. 

\begin{thm}[{\cite[Propositions~1 and~2 and the proof of 
Proposition~4]{KapustkaRanestad2013}}]
\label{thm:KaRa} $\,$
\begin{enumerate}
\renewcommand\labelenumi{\rm (\alph{enumi})}
\renewcommand\theenumi{\rm (\alph{enumi})}
\item
\label{thm:KaRa-a} 
Let $\Sigma(V)$ be
the Hilbert scheme of lines in $V$. Then $\Sigma(V)$ is isomorphic to
the variety $\Fl(\PP^2)$ of full flags on $\PP^2$.

\item
\label{thm:KaRa-b} 
Let $\SSS(V)$ be
the Hilbert scheme of cubic scrolls in $V$. Then $\SSS(V)$
is isomorphic to
a disjoint union of two projective planes.
\end{enumerate}
\end{thm}

This theorem and the next lemma show that the cubic scrolls in $V$ play the same role as do the planes in a smooth quadric fourfold, 
cf.~\cite[Lecture 22]{Harris:book}.

\begin{lem}[{\cite[Proposition~9.6]{PZ18}}]
\label{lemma:intersection}
Let $\SSS_1$ and $\SSS_2$ be the connected components of $\SSS(V)$. Then for any $S_i\in\SSS_i$ and $S_j\in\SSS_j$, $i,j\in\{1,2\}$ 
we have the following relation: $S_i\cdot S_j=\updelta_{i,j}$ in $H^*(V,\ZZ)$.
\end{lem}

\begin{lem}[{\cite[Lemma~9.2, Corollary~4.5.2]{PZ18}}]$\,$
\label{lemma:AS}
\begin{enumerate}
\item
\label{lemma:AS-i}
For any cubic scroll $S$ on $V$ 
there exists a unique hyperplane section $A_S$ of $V$ such that $\Sing(A_S)=S$.
This hyperplane section coincides with the union of lines on $V$ meeting $S$, and any line contained in $A_S$ meets $S$.
\item
\label{lemma:AS-ii}
For any point $P\in A_{S}\setminus S$ there is exactly one line $l\subset A_{S}$ passing through $P$
and meeting $S$. Such a line meets $S$ in a single point. 

\item
\label{lemma:AS:1} 
If $S\subset V$ is a cubic cone, then $V\setminus A_S\cong\A^4$.
\end{enumerate}
\end{lem}

\begin{lem}[{\cite[Proposition~8.2]{PZ18}}]
\label{lemma:lines}
Let $v\in V$ be a point.
There exists a divisor $\EuScript{B}\subset V$ such that 
the following hold:
\begin{enumerate}
\item
\label{lemma:lines:1}
for any point $v\in V\setminus \EuScript{B}$ there are exactly three lines passing through $v$;
\item
\label{lemma:lines:2}
for $v\in \EuScript{B}$ the number of lines passing through $v$ either is infinite, or 
equals~$2$ or~$1$; 
\item
\label{lemma:lines:3}
if the number of lines passing through $v$ is infinite, then 
the union of these lines is a cubic cone $S$ with vertex $v$, and $S\subset\EuScript{B}$. 
\end{enumerate}
\end{lem}

\begin{lem}[{\cite[Lemmas~9.3, 9.4, 9.9(a), Corollaries~9.7.3, 9.7.4, 9.10.1]{PZ18}}] $\,$
\begin{enumerate}
\label{lemma:lines-bis}
\item
\label{lemma:lines:4}
Any line on $V$ can be contained in at most finite number of cubic scrolls, and in at most two cubic cones;
\item
\label{lemma:lines:5} 
two cubic cones from different components of $\SSS(V)$ either are disjoint, or share a unique common ruling. Two cubic cones from the same component $\SSS_i$ meet transversally in a single point different from their vertices;
\item
\label{lemma:lines:6} 
if two cubic cones $S$ and $S'$ are disjoint, then $v_S\notin A_{S'}$, where $v_S$ stands for the vertex of $S$;
\item
\label{lemma:lines:7} 
if two cubic cones $S$ and $S'$ share a common ruling $l$, then $S\cup S'$ coincides with the union of lines on $V$ meeting $l$.
\end{enumerate}
\end{lem}

\begin{lem}[{\cite[Lemma~9.7.2, Corollary~9.7.3(i)]{PZ18}}] 
\label{lem-splitting-lines} 
In the notation of Theorem~\xref{thm:KaRa}\ref{thm:KaRa-a} the Hilbert scheme of lines $\Sigma(V)=\Fl(\PP^2)$ is a divisor of type $(1,1)$ on $\PP^2\times (\PP^2)^\vee$. Let 
\begin{equation}
\label{eq:pr}
\pr_1: \Sigma(V)=\Fl(\PP^2)\longrightarrow \PP^2\quad 
\text{and} \quad \pr_2: \Sigma(V)=\Fl(\PP^2)\longrightarrow (\PP^2)^\vee
\end{equation} 
be the natural projections. 
Any line $l$ on $V$ is the common ruling of exactly two cubic scrolls, say, $S_1(l)\in\SSS_1$ and $S_2(l)\in\SSS_2$. 
Hence, there are well-defined morphisms
\[
\pr_i\colon \Sigma(V)\xlongrightarrow{\hspace{2em}} \SSS_i=\PP^2,\quad l\xmapsto{\hspace{2em}} S_i(l),\quad i\in\{1,\, 2\}
\]
which coincide with the ones in \eqref{eq:pr}. The fiber of $\pr_i$ over $S\in\SSS_i$ is the line on $\Sigma(V)$ which parameterizes the rulings of $S$.
\end{lem}

Recall that the Fano-Mukai fourfolds $V=V_{18}$ are classified in three types according to the group $\Aut^0(V)$,
which can be isomorphic to one of the following groups
\[\Gm^2,\quad \Ga\times\Gm,\quad \GL_2(\CC).\]
This classification reflects the geometry of $V$, namely, the number of cubic cones on $V$.

\begin{lem} 
\label{lem:cubic-cones:2} 
Let $\SSS_i\cong \PP^2$ be a connected component of $\SSS(V)$.
Then the following hold. 
\begin{enumerate} 
\item
\label{lem:cubic-cones:2:cycle} 
If $\Aut^0(V)=\Gm^2$ then $\SSS_i$ contains exactly $3$ cubic cones for $i=1,2$. The six cubic cones in $V$ form a cycle so that the neighbors have a common ruling and belong to different components of $\SSS(V)$, and the pairs of opposite vertices of the cycle correspond to the pairs of disjoint cubic cones.
\item
\label{lem:cubic-cones:2:iii} 
If $\Aut^0(V)=\Ga\times\Gm$ then $\SSS_i$ contains exactly $2$ cubic cones for $i=1,2$. The four cubic cones in $V$ form a chain, that is, the neighbors have a common ruling and belong to different components of $\SSS(V)$, and the pair of extremal vertices corresponds to the unique pair of disjoint cones.
\item
\label{lem:cubic-cones:2:i} 
If $\Aut^0(V)=\GL_2(\CC)$ then the subfamily of cubic cones in $\SSS_i\cong \PP^2$ consists of a projective line and an isolated point.  The cubic cones $S_i\in \SSS_i$, $i=1,2$, represented by these isolated points are disjoint, and this is the only pair of $\Aut^0(V)$-invariant cubic cones on $V$.
\end{enumerate}
\end{lem}

\begin{proof}
By Lemma~\ref{lem-splitting-lines} the variety $\Lambda(S)\subset\Sigma(V)$ of rulings of a cubic scroll $S\subset V$
is the fiber of one of the projections $\pr_i$, so it is a line under the Segre embedding $\Sigma(V)\subset\PP^2\times(\PP^2)^\vee\hookrightarrow\PP^8$, and any line on $\Sigma(V)$ appears in this way. A line $l$ on $V$ is called a \emph{splitting line} if the union of lines on $V$ meeting $l$ splits into a union of two cubic scrolls. Assuming $\Aut^0(V)\neq \GL_2(\CC)$ the subvariety $\Sigma_{\s}(V)\subset\Sigma(V)$ of splitting lines is a del Pezzo sextic, which admits two birational contractions to $\PP^2$
\cite[Proposition 10.2]{PZ18}.
The scroll $S$ on $V$ is a cubic cone exactly when the line $\Lambda(S)$ lies on $\Sigma_{\s}(V)$
\cite[Proposition~9.10]{PZ18}. The surface $\Sigma_{\s}(V)$ is smooth in case~\ref{lem:cubic-cones:2:cycle},
and has a unique node in case~\ref{lem:cubic-cones:2:iii} \cite[Proposition 10.2]{PZ18}. 
It is well known that a smooth sextic del Pezzo surface contains exactly $6$ lines, and these are arranged in a cycle. 
If $\Sigma_{\s}(V)$ has a singularity of type $A_1$, then it contains exactly $4$ lines, and these are
arranged in a chain.

In the both cases, a pair of intersecting lines on $\Sigma_{\s}(V)$ corresponds to a pair of cubic cones on $V$ sharing a common ruling. By Lemma~\ref{lemma:lines-bis}\ref{lemma:lines:7}, such cones belong to distinct components of $\SSS(V)$. Thus, two neighbors of the same cubic cone belong to the same component $\SSS_i$ and meet transversally in a unique point, and two cubic cones separated by two others belong to distinct components of $\SSS(V)$ and are disjoint, see Lemma~\ref{lemma:lines-bis}\ref{lemma:lines:5}. This gives~\ref{lem:cubic-cones:2:cycle} and~\ref{lem:cubic-cones:2:iii}. See \cite[Corollary~10.3.2]{PZ18} for~\ref{lem:cubic-cones:2:i}.
\end{proof}

\begin{cor}[{\cite[Lemmas~12.2 and~12.8.1]{PZ18}}]
\label{cor:invariant-cone}
If $\Aut^0(V)=\Gm^2$ or $\Ga\times\Gm$ then any cubic cone in $V$ is $\Aut^0(V)$-invariant. 
If $\Aut^0(V)=\GL_2(\CC)$ then $\SSS_i$ contains exactly one $\Aut^0(V)$-invariant cubic cone.
\end{cor}

\begin{lem}
\label{lemma:capAS} 
Let $\SSS_1\subset \SSS(V)$ be a connected component.
Then the set $\bigcap_{S\in \SSS_1} A_S$ coincides with the union of vertices of cubic cones in $\SSS_1$.
In particular, $\bigcap_{S\in \SSS(V)} A_S=\varnothing$.
\end{lem}

\begin{proof} Let $v\in V$ be the vertex of a cubic cone $S_v\in \SSS_1$. 
By Lemma~\ref{lemma:intersection} one has $S\cap S_v\neq\varnothing$ for any $S\in\SSS_1$, 
and so, $v\in \bigcap_{S\in \SSS_1} A_S$ due to Lemma~\ref{lemma:AS}.

Conversely, let $v\in \bigcap_{S\in \SSS_1} A_S$.
Assume to the contrary that the lines on $V$ passing through $v$ form a finite set, say, $\{l_1,\dots, l_k\}$. 
By Lemmas~\ref{lemma:lines}\ref{lemma:lines:3} and~\ref{lemma:lines-bis}\ref{lemma:lines:4}, 
the number of lines on $V$ passing through a general point of $l_i$ is finite. 
Hence, the family $\Sigma(V;l_1,\dots, l_k)$ of lines in $V$ meeting 
$\bigcup_{i=1}^k l_i$ is one-dimensional, and again by Lemma~\ref{lemma:lines-bis}\ref{lemma:lines:4},
any line in this family
is contained in a finite number of cubic scrolls. It follows that the subfamily $\SSS_1(l_1,\dots, l_k)$ of cubic scrolls from $ \SSS_1$
meeting $\bigcup_{i=1}^k l_i$ is one-dimensional too. 
Since $\dim\SSS_1=2$, see Theorem~\ref{thm:KaRa}\ref{thm:KaRa-b}, one has $\SSS_1(l_1,\dots, l_k)\neq \SSS_1$, 
and the general cubic scroll 
$S\in \SSS_1$ does not meet $\bigcup_{i=1}^k l_i$. This implies
$v\notin A_S$, a contradiction.

Thus, any point $v\in \bigcap_{S\in \SSS_1} A_S$ is the vertex of a cubic cone, say, $S_v$. 
Then one has $S_v\cap S\neq \varnothing$ for any $S\in \SSS_1$.
Assuming $S_v\notin \SSS_1$ it follows from Lemma~\ref{lemma:intersection} that for any $S \in \SSS_1$ the intersection $S_v\cap S$ contains a curve. Then any line in $S_v$ meets $S$, and so, $S_v \subset 
\bigcap_{S\in \SSS_1} A_S$. By the preceding, any point of $S_v$ is a vertex of a cubic cone in $V$. This contradicts Lemma~\ref{lemma:lines-bis}\ref{lemma:lines:4} and proves the first assertion. The second assertion follows from the first since two distinct cubic cones cannot share the same vertex, see Lemma~\ref{lemma:lines}\ref{lemma:lines:3}.
\end{proof}

For instance, in the case $\Aut^0(V)=\GL_2(\CC)$ the intersection $\bigcap_{S\in \SSS_i} A_S$ 
is the twisted cubic $\Gamma_i$, $i=1,2$ as in Lemma~\ref{lemma:2-cones} below, see \cite[Theorem~13.5(b)]{PZ18}. 

\section{$\Aut^0(V)$-action on the Fano-Mukai fourfold $V$ of genus $10$}
\label{sec:Aut-action}
\subsection*{Generalities}

\begin{lem}
\label{lem:Hilbert-schemes} 
\begin{enumerate}
\item
\label{lem:Hilbert-schemes-1} 
Under the induced $\Aut^0(V)$-action on the Hilbert scheme of lines $\Sigma(V)$, the stabilizer of the general point is trivial.
\item
\label{lem:Hilbert-schemes-2} 
Under the induced $\Aut^0(V)$-action on the Hilbert scheme
of cubic scrolls $\SSS(V)$, the stabilizer of the general point is finite.
\end{enumerate}
\end{lem}

\begin{proof}
Statement~\ref{lem:Hilbert-schemes-1} follows from the fact that through the general point of $V$ 
pass at least two lines, see Lemma~\ref{lemma:lines}\ref{lemma:lines:1}, and two lines meet in a single point. 
In turn,~\ref{lem:Hilbert-schemes-2} follows from the fact that through the general line $l$ on $V$ 
pass a finite number of cubic scrolls, and $l$ is a component of the intersection of these scrolls, 
see Lemma~\ref{lemma:lines-bis}\ref{lemma:lines:4}. 
\end{proof}

\begin{lem} 
\label{lemma:2-cones} 
Let $S_1,\, S_2\subset V$ be disjoint cubic cones, and let $\Gamma_1=S_1\cap A_{S_2}$,
$\Gamma_2=S_2\cap A_{S_1}$.
Then the following holds.
\begin{enumerate}
\item
\label{lemma:2-cones:1} 
$\Gamma_1$ and $\Gamma_2$ are rational twisted cubic curves;
\item
\label{lemma:2-cones:2}
there exists a one-parameter family of lines $l_t$ in $V$ joining $\Gamma_1$ and $\Gamma_2$;
\item
\label{lemma:2-cones:2a}
any line in $V$ passing through a point $\gamma\in \Gamma_i$ 
is either a ruling of $S_i$, or a unique member of the family $l_t$; 
\item
\label{lemma:2-cones:3}
$D:=\bigcup_{t\in\PP^1} l_t$ is a rational normal sextic scroll contained in $A_{S_1}\cap A_{S_2}$;
\item
\label{lemma:2-cones:4}
If $S_1$ and $S_2$ are $\Aut^0(V)$-invariant then $D$ and the curves $\Gamma_i$, $i=1,2$, 
are as well, and the stabilizer of the general point of $D$ in $\Aut^0(V)$ is finite.
\end{enumerate}
\end{lem}
\begin{proof}
Since $S_1\cap S_2=\varnothing$ one has $v_i\notin A_{S_j}$ for $i,j=1,2$, $i\neq j$, 
see Lemma~\ref{lemma:lines-bis}\ref{lemma:lines:6}. This yields~\ref{lemma:2-cones:1}.
By Lemma~\ref{lemma:AS}\ref{lemma:AS-i}--\ref{lemma:AS-ii} for any $\gamma\in \Gamma_1$ 
there exists a unique line $l_\gamma\subset A_{S_1}\cap A_{S_2}$ joining $\gamma$ and $S_2$. 
This line $l_\gamma$ meets $\Gamma_2=S_2\cap A_{S_1}$. This shows~\ref{lemma:2-cones:2} 
and the inclusion $D\subset A_{S_1}\cap A_{S_2}$, where $D$ is as in~\ref{lemma:2-cones:3}.
Statement~\ref{lemma:2-cones:2a} follows from 
Lemma~\ref{lemma:lines}\ref{lemma:lines:1}--\ref{lemma:lines:3}.
For the first assertion in~\ref{lemma:2-cones:3} see, e.g., \cite[Example~8.17]{Harris:book}. 
The first assertion in~\ref{lemma:2-cones:4} is immediate. To prove the second, we assume $\Aut^0(V)$ to be abelian. 
In the case $\Aut^0(V)=\GL_2(\CC)$ one can either restrict to the maximal torus of $\GL_2(\CC)$, or simply deduce the result from 
\cite[Theorem~13.5]{PZ18}.
Suppose $G$ is 
a one-parameter subgroup of $\Aut^0(V)$ acting trivially on $D$.
Then $G$ fixes the general line meeting $D$. So, there is a two-dimensional subvariety 
$\Sigma'\subset \Sigma(V)$ parameterizing
$G$-invariant lines on $V$. By the description of $\Sigma(V)$, see 
Theorem~\ref{thm:KaRa}\ref{thm:KaRa-a} and 
Lemma~\ref{lem-splitting-lines}, 
the group $G$ acts trivially on $\Sigma(V)$ and on $V$, a contradiction.
\end{proof}

\begin{lem} 
\label{lemma:stabilizers-of-cones} 
Let $S$ be an $\Aut^0(V)$-invariant cubic cone in $V$, and let $G_0$ be
a two-dimensional connected abelian subgroup of $\Aut^0(V)$. Then $G_0$ acts on $S$ with an open orbit. 
\end{lem}
\begin{proof} 
We repeat the argument from the proof of the last statement 
in Lemma~\ref{lemma:2-cones}\ref{lemma:2-cones:4}. By Lemmas~\ref{lemma:AS}\ref{lemma:AS-ii} and~\ref{lemma:lines}\ref{lemma:lines:2}--\ref{lemma:lines:3} through the general point $P$ of $S$ passes a unique line different from the ruling of $S$ through $P$. Hence, the family of lines on $V$ meeting $S$ is two-dimensional. If the $G_0$-action on $S$ does not have an open orbit, then all the lines in this family have a common one-dimensional stabilizer $G$. As before, this stabilizer acts trivially on $\Sigma(V)$ and on $V$, a contradiction.
\end{proof}

\subsection*{The fixed points of the torus.}
\label{sit:T-fixed-points} 
Consider the simple complex algebraic group $G_2$ of rank 2 and of dimension 14. By the Mukai construction \cite{Muk89} (see also 
\cite[Theorem~7.1]{PZ18}), any Fano-Mukai fourfold $V$ of genus $10$ is a hyperplane section of the homogeneous fivefold $\Omega=G_2/P\subset\PP^{13}$, where $P\subset G_2$ is a parabolic subgroup of dimension $9$ corresponding to a long root, and so, $\Omega$ is the corresponding adjoint variety. It is known \cite[Theorem~1.2]{PZ18} that $\Aut^0(V)$ is the stabilizer of $V$ in $G_2$ acting naturally on $\Omega$.

\begin{lem}
\label{lem:torus-fixed-points} 
Let $T\subset G_2$ be a maximal torus. Then $T$ has exactly six fixed point in $\Omega$. 
\end{lem}

\begin{proof} 
The left coset $\omega=gP\in\Omega$, where $g\in G_2$, is a fixed point of $T$ if and only if for any $t\in T$ there exists $p\in P$ with $tg=gp$, that is, $g^{-1}Tg\subset P$. Since the maximal tori in $G_2$ are conjugated and $P$ contains the Borel subgroup of $G_2$, we may assume $T\subset P$. 
Consider the following subgroup of $G_2$:
\[
\Gamma:= \{ g\in G_2 \mid g^{-1}Tg\subset P\}.
\]
Clearly, $\Gamma\supset P$. Two $T$-fixed points $\omega_i=g_iP\in\Omega$, $i=1,2$, coincide if and only if $g_2\in g_1P$. So, the fixed points of $T$ in $\Omega$ are in one-to-one correspondence with the elements of the left coset space $\Gamma/P$.

We have $N_{G_2}(T)\subset\Gamma$, where $N_G(H)$ stands for the normalizer subgroup of a subgroup $H\subset G$. The normalizer $N_{G_2}(T)$ acts on $\Gamma/P$ via
\[N_{G_2}(T)\ni n: gP\mapsto ngP\quad \forall g\in\Gamma.\]
This action is transitive. Indeed, given $g\in\Gamma$, let $T'=g^{-1}Tg\subset P$. The maximal tori $T$ and $T'$ are conjugate in $P$, and so, $T'=p^{-1}Tp$ for some $p\in P$. Then $n:=pg^{-1}\in N_{G_2}(T)$ verifies $nP=gP$. The stabilizer of the coset $P$ in $N_{G_2}(T)$ under this action is $P\cap N_{G_2}(T)=N_P(T)$. 
Therefore, 
the fixed points of $T$ in $\Omega$ are in one-to-one correspondence with the elements of the 
left coset space
\[
N_{G_2}(T)/N_{P}(T)
=\W(G_2)/\W(P)=\mathfrak{D}_6/\{\pm 1\} = \ZZ/6\ZZ,
\] 
where $\W(G_2)=N_{G_2}(T)/T$ and $\W(P)=N_P(T)/T$ stand for the Weyl groups of $G_2$ and of $P$, respectively, and $\mathfrak{D}_n$ is the $n$th dihedral group. This yields the assertion.
\end{proof}

\subsection*{The case $\Aut^0(V)=\Gm^2$}
In this subsection we prove Theorems~\ref{thm:main} and 
\ref{mthm} in the case $\Aut(V)=\Gm^2$, that is, we let $V=V^{\g}_{18}$. 

\begin{sit}
\label{sit:1} 
The Fano-Mukai fourfold $V$ with $\Aut^0(V)=\Gm^2$ contains exactly six 
cubic cones, see Lemma~\ref{lem:cubic-cones:2}\ref{lem:cubic-cones:2:cycle}.
By Lemma~\ref{lemma:lines}\ref{lemma:lines:3}, any cubic cone $S\subset V$ coincides 
with the union of lines in $V$ passing through its vertex, is
$\Aut^0(V)$-invariant, and
its vertex is fixed under $\Aut^0 (V)$. By Lemma~\ref{lemma:lines}\ref{lemma:lines:3}, 
the vertices of distinct cubic cones are distinct. 
Using Lemma~\ref{lem:torus-fixed-points} 
we deduce the following corollary. 
\end{sit}

\begin{cor}
\label{cor:fixed-points} 
Let $\Aut^0(V)=\Gm^2$. Then the vertices $v_i$ of the cubic cones $S_i$, $i=1,\ldots,6$, 
are the only fixed points of the torus $T=\Aut^0(V)$ acting on $V$. 
\end{cor}

The next corollary yields Theorem~\ref{thm:main}\ref{thm:main-b}.

\begin{cor}
\label{cor:Ai-intersection} 
One has $\bigcap_{i=1}^6 A_{S_i}=\varnothing$.
\end{cor}

\begin{proof}
Assume the contrary holds. Then by the Borel fixed point theorem, the intersection 
$\bigcap_{i=1}^6 A_{S_i}$ contains a fixed point of the torus $\Aut^0(V)=\Gm^2$. By Corollary~\ref{cor:fixed-points} this point is the vertex of a cubic cone, say $S_1\in \SSS_1$.
By Lemma~\ref{lemma:lines-bis}\ref{lemma:lines:6}, $S_1$ meets any of the remaining cones $S_i$, $i=2,\dots,6$.
In particular, $S_1$ meets the cones, say $S_2$, $S_4$, and $S_6$, which belong to the other component $\SSS_2$ of $\SSS(V)$.
Then $S_1\cdot S_i=0$ for $i=2,4,6$, see Lemma~\ref{lemma:intersection}. By Lemma~\ref{lemma:lines-bis}\ref{lemma:lines:5}, $S_1$ shares a common ruling with $S_2$, $S_4$, and $S_6$. 
These three rulings of $S_1$ are $\Aut^0(V)$-invariant and pairwise distinct, see \cite[Corollary~9.7.3(ii)]{PZ18}. It follows that the $\Aut^0(V)$-action on the base of $S_1$ is trivial, and so, any ruling of $S_1$ is invariant. 
However, this contradicts the fact that the torus $\Gm^2$ acts on $S_1$ with an open orbit, see Lemma~\ref{lemma:stabilizers-of-cones}. 
\end{proof}

\begin{rem} An alternative argument is as follows.
By Lemma~\ref{lem:cubic-cones:2}~\ref{lem:cubic-cones:2:cycle} 
one has $S_1\cap S_4=\varnothing$, and so, $v_{S_1}\notin A_{S_4}$ due to Lemma~\ref{lemma:lines-bis}\ref{lemma:lines:6}. This gives again a contradiction. 
\end{rem}

\begin{proof}[Proof of Theorem~\xref{mthm} in the case $\Aut(V)=\Gm^2$] 
Recall that the affine space $\A^4$ is a flexible variety, see, e.g., \cite[Lemma~5.5]{KZ99}.
By Proposition~\ref{prop:intersection-of-Ai}, $V$ is covered 
by the flexible Zariski open subsets $U_i=V\setminus A_{S_i}\cong \A^4$, $i=1,\ldots,6$. Thus, the criterion of Theorem~\ref{thm:criterion-MPS}\ref{thm:criterion-MPS:2} applies and gives the result. 
\end{proof}

\subsection*{The case $\Aut^0(V)=\Ga\times\Gm$}
In this subsection we let $V=V^{\aaa}_{18}$.

\begin{lem}
\label{lemma:matrix}
Any effective action of $\Ga\times \Gm$ on $\PP^2$ 
can be given in suitable coordinates by the matrices
\begin{equation}
\label{eq:matrix}
\begin{pmatrix}
\lambda & 0 & 0
\\
0 &1 & \mu
\\
0& 0 & 1
\end{pmatrix} 
\quad\text{
where}\quad\lambda \in \CC ^*,\,\, \mu\in \CC.
\end{equation}
This action has exactly two invariant lines $J_1,\, J_2\subset \PP^2$
and exactly two fixed points $P_0\in J_1\cap J_2$ and $P_1\in J_2\setminus J_1$, 
see Fig.~\xref{fig}.
\end{lem}

\begin{proof} 
Since the group $\Ga\times \Gm$ is abelian and acts effectively on $\PP^2$, \eqref{eq:matrix} is the only possibility 
for the Jordan normal form of its elements modulo scalar matrices.
\end{proof}

\begin{figure}[h]
\begin{tikzpicture}
\draw[name path=DD1, thick] (0,0) node[left,yshift=5pt, xshift=0pt]{$\scriptstyle J_1$}-- (3,1) ; 
\draw[name path=DD2, thick] (2,1) -- (5,0) node[right,yshift=5pt, xshift=0pt]{$\scriptstyle J_2$};
\draw[name path=DD3, white] (4.5,-0.5) -- (4.5,1) ;
\fill[name intersections={of=DD1 and DD2}] (intersection-1) circle(2.5pt);
\node at (intersection-1)[above, yshift=0.3em] {$\scriptstyle P_0$};
\fill[name intersections={of=DD2 and DD3}] (intersection-1) circle(2.5pt) ;
\node at (intersection-1)[above, xshift=7pt] {$\scriptstyle P_1$};
\end{tikzpicture}
\caption{}
\label{fig}
\end{figure}

\begin{prop}
\label{prop:211} 
For $V=V^{\aaa}_{18}$
the following assertions hold.
\begin{enumerate}
\item
\label{prop:211-1}
The action of $\Aut^0(V)=\Ga\times \Gm$ on each component
$\SSS_i\cong \PP^2$ of $\SSS(V)$, $i=1,2$ is given by matrices \eqref {eq:matrix},
and the action on $\Sigma(V)\cong \Fl(\PP^2)$ is the induced one.
\item
\label{prop:211-2}
The subfamily $\SSS_1^{\aaa}$ of $\Ga$-invariant cubic scrolls corresponds to the line $J_2$ 
on $\SSS_1=\PP^2$, and
the subfamily $\SSS_1^{\m}$ of $\Gm$-invariant cubic scrolls corresponds to $J_1$, 
see Fig.~\xref{fig}.

\item
\label{prop:211-4}
There are exactly three $\Aut^0(V)$-invariant lines $l_{i,i+1}$, $i\in \{1,\, 2,\, 3\}$ and exactly four $\Aut^0(V)$-invariant cubic scrolls $S_i$, $i=1,\dots,4$ on $V$. 
With a suitable enumeration, $l_{i,i+1}$ is a common ruling of $S_i$ and $S_{i+1}$, while $S_i$ and $S_j$ have no common ruling if $j-i\neq \pm 1$. 
Furthermore, $S_i$ and $S_j$ belong to the same connected component of $\SSS(V)$ if and only if $j\equiv i\mod 2$.

\item
\label{prop:211-3}
Any cubic scroll $S_i$ in~\ref{prop:211-4} is a cubic cone, and any cubic cone on $V$ 
coincides with one of the $S_i$'s.

\item
\label{prop:211-5}
There are exactly four $\Aut^0(V)$-fixed points $v_1,\dots,v_4$ on $V$. These points are the vertices of cubic cones $S_1,\dots,S_4$. 
Furthermore, one has
\[
v_2=l_{1,2}\cap l_{2,3}=S_1\cap S_3,\qquad v_3=l_{2,3}\cap l_{3,4}=S_2\cap S_4,\qquad
S_1\cap S_4=\varnothing.
\] 
The three lines $l_{i,i+1}$ form a chain.
\item
\label{cor:invariant-on-Va-4}
There are exactly two families of $\Ga$-invariant lines on $V$. These are the families of rulings of $S_2$ and $S_3$.
\item
\label{cor:invariant-on-Va-5}
There are exactly three families of $\Gm$-invariant lines on $V$. These are the families of rulings of $S_1$, $S_4$, and of the smooth sextic scroll $D\subset V$, see \xref{sit:Ga-Gm-cones}. 
\end{enumerate}
\end{prop}

Collecting the information from Proposition~\ref{prop:211} we see that 
the configuration of the cones $S_i$ looks as the one on Fig~\ref{fig2}.

\begin{figure}[ht]
\begin{tikzpicture}
\draw[fill=black] (-3,4) circle (3pt) node[above=3pt]{$v_2$};
\draw[fill=black] (3,4) circle (3pt) node[above=3pt]{$v_4$};
\draw[fill=black] (-6,0) circle (3pt) node[below=3pt]{$v_1$};
\draw[fill=black] (0,0) circle (3pt) node[below=3pt]{$v_3$};
\draw[thick] plot [smooth,tension=1] coordinates{ (-3,4)(0,4.1)(3,4)};
\draw[thick] plot [smooth,tension=1] coordinates{ (-6,0)(-3,-0.2)(0,0)};
\draw[thick] (-3,4) -- node[above=5pt]{$l_{1,2}$} (-6,0) ;
\draw[very thick] (-3,4) -- node[below=5pt]{$l_{2,3}$} (0,0) ; 
\draw[thick] (0,0) -- node[above=5pt]{$l_{3,4}$} (3,4);
\node at (-6.3,2.5) {$S_1$};
\node at (-0.3,2.5) {$S_3$};
\node at (-3.3,1) {$S_2$};
\node at (3,1.3) {$S_4$};
\draw[thick] plot [smooth,tension=0.3] coordinates{(-6,0) (-7.9,3.9)(-3,4)};
\draw[thick] plot [smooth,tension=0.3] coordinates{(0,0) (6,0.8)(3,4)}; 
\end{tikzpicture}
\caption{}
\label {fig2}
\end{figure}

\begin{nota}
\label{sit:Ga-Gm-cones} 
There is exactly one pair of disjoint cubic cones on $V$, namely, $(S_1,S_4)$.
We let $D\subset V$ be the smooth sextic scroll as 
in Lemma~\ref{lemma:2-cones}\ref{lemma:2-cones:3}, 
that is, the union of lines on $V$ joining the corresponding points of the rational twisted cubic curves 
$\Gamma_1=S_1\cap A_{S_4}$ and $\Gamma_4=S_4\cap A_{S_1}$.
Notice that the curves $\Gamma_1$ and $\Gamma_4$, as well as the surface $D$ are
$\Aut^0(V)$-invariant.
\end{nota}

\begin{proof}[{Proof of Proposition~\xref{prop:211}}] 
Assertions~\ref{prop:211-1} and~\ref{prop:211-2} are immediate, see Lemma~\ref{lemma:matrix}. 
Assertion~\ref{prop:211-4} follows from Lemmas~\ref{lem-splitting-lines} and~\ref{lemma:matrix}.

\ref{prop:211-3}. 
Any cubic cone on $V$ is $\Aut^0(V)$-invariant, hence, 
is contained in $\{S_1,\dots,S_4\}$, see Corollary~\ref{cor:invariant-cone}.
We have to show only that any invariant cubic scroll is a cubic cone. 
Actually, this follows from Lemma~\ref{lem:cubic-cones:2}\ref{lem:cubic-cones:2:iii}, however,
we provide an alternative proof. 

Suppose to the contrary $S_2$ is smooth,
and so $S_2\cong \FF_1$. Then its exceptional section is an invariant line,
hence, it coincides with $l_{3,4}$. 
By \cite[Proposition~9.5]{PZ18} the corresponding connected component of $\SSS(V)$ contains a cubic cone, which 
must coincide with $S_4$ due to~\ref{prop:211-4}. Let
$l$ be the ruling of $S_2$ passing through the vertex $v_4$ of $S_4$. Then $l$ is a common ruling
of $S_2$ and $S_4$. This contradicts~\ref{prop:211-4}. 
Thus, $S_2$ is a cubic cone. By symmetry, $S_3$ is a cubic cone too.

Suppose further $S_1$ is smooth. 
If $S_1$ contains the vertex $v_3$ of $S_3$, then $S_1$ and $S_3$ have a common ruling,
a contradiction. Thus, $v_3\notin S_1$, and so, $S_1$ does not contain the invariant lines 
$l_{2,3}$ and $l_{3,4}$. It follows that the ruling $l_{1,2}$ of $S_1$ is the only $\Aut^0(V)$-invariant
line on $S_1$, which coincides then with the exceptional section of $S_1$, a contradiction.
Thus, $S_1$ is a cubic cone, and by symmetry, the same holds for $S_4$.

\ref{prop:211-5} 
Notice that the vertices of the cubic cones in $V$ are fixed by $\Aut^0(V)$, 
because these cones are invariant. Let $v$ be a fixed point of the $\Aut^0(V)$-action on $V$ 
different from the vertices of cubic cones. According to Lemma~\ref{lemma:lines}\ref{lemma:lines:3}, 
the number of lines on $V$ passing through $v$ is finite. Hence each of these lines is $\Aut^0(V)$-invariant. 
By~\ref{prop:211-4} such a line coincides with the common ruling $l_{i,i+1}$ of a pair $(S_i,\,S_{i+1})$, $i\in\{1,2,3\}$. 
However, any line $l_{i,i+1}$ contains exactly two $\Aut^0(V)$-fixed points, namely, the vertices $v_i$ and 
$v_{i+1}$. The remaining statements are immediate. 

\ref{cor:invariant-on-Va-4} and~\ref{cor:invariant-on-Va-5}.
By Theorem~\ref{thm:KaRa}\ref{thm:KaRa-a} we have $\Sigma(V)\cong \Fl(\PP^2)$.
By Lemma~\ref{lemma:matrix} the group $\Aut^0(V)\cong \Ga\times \Gm$
acts on $\Sigma(V)\cong \Fl(\PP^2)$ via \eqref{eq:matrix}.
Looking at Fig.\ref{fig} one can select all the lines on $V$ with one-dimensional stabilizers. 
There are exactly five such families of lines; they correspond to the following five families of flags on $\PP^2$: 
\begin{itemize}
\item
$(P_0, l)$, where $l$ runs over the pencil of lines through $P_0$;
\item
$(P_1, l)$, where $l$ runs over the pencil of lines through $P_1$;
\item
$(P, J_2)$, where $P$ runs over $J_2$;
\item
$(P, J_1)$, where $P$ runs over $J_1$;
\item
$(P, l)$, where $P$ runs over $J_1$ and $l$ passes through $P_1$.
\end{itemize}

The last family corresponds to the family of rulings of $D$, and the other four correspond 
to the families of rulings of the cubic cones $S_1,\ldots,S_4$. 
\end{proof}

\section{Affine $4$-spaces in $ V^{\aaa}_{18}$ and flexibility of affine cones}
\label{sec:general-case}

\subsection*{Affine $4$-spaces in $ V^{\aaa}_{18}$}
In this subsection we 
analyze affine charts isomorphic to $\A^4$ on the Fano-Mukai fourfold $V=V^{\aaa}_{18}$ 
of genus $10$ with $\Aut^0(V)=\Ga\times\Gm$, and provide a modified criterion of flexibility 
of affine cones over projective varieties in terms of existence of certain cylinders. 

The following proposition proves the first part of 
Theorem~\ref{thm:main}\ref{thm:main-c}; the second part will be proven in Proposition~\ref{prop:A2-cylinders}.

\begin{prop}
\label{prop:intersection-of-Ai} 
In the notation of Proposition~\xref{prop:211}
one has $\bigcap_{j=1}^4 A_{S_j}=l_{2,3}$.
\end{prop}

The proof is done below. We need the following auxiliary facts.
Let $Y=\bigcap_{j=1}^4 A_{S_j}$. Clearly, $Y$ is $\Aut^0(V)$-invariant.

\begin{claim}
\label{claim:SS0} One has
$l_{2,3}\subset Y$.
\end{claim}
\begin{proof}
The line $l_{2,3}$ intersects all the $S_i$. Hence, $l_{2,3}\subset A_{S_i}$
for all $i$.
\end{proof}

\begin{claim}
\label{claim:SS1} 
$S_1\cap Y=\{v_2\}$ and $S_4\cap Y=\{v_3\}$.
\end{claim}
\begin{proof}
We have $v_1\notin A_{S_4}$ because $S_1\cap S_4=\varnothing$.
Hence $S_1\cap A_{S_4}$ is a smooth irreducible hyperplane section of the cone $S_1$.
Since $v_1\in A_{S_2}$, the intersection
$S_1\cap A_{S_2}$ is a singular hyperplane section of $S_1$. Therefore,
the intersection $S_1\cap Y$ is a finite set, whose points are fixed by $\Aut^0(V)$.
Then by Proposition~\ref{prop:211}\ref{prop:211-5} and Claim~\ref{claim:SS0} we have $S_1\cap Y=\{v_2\}$.
\end{proof}

\begin{claim}
\label{claim:SS2} 
$S_2\cap Y=S_3\cap Y=l_{2,3}$.
\end{claim}
\begin{proof}
We have $v_2\in S_2\cap Y$ and $S_2\not\subset Y$ because $l_{1,2} \not\subset Y$ by Claim~\ref{claim:SS1}.
Since $S_2\cap Y=S_2\cap \langle Y\rangle$
the intersection $S_2\cap Y$ consists of a finite number of rulings of $S_2$.
These rulings are $\Aut^0(V)$-invariant.
Since $l_{1,2}\not \subset Y$ by Claim~\ref{claim:SS1}
the only possibility is $S_2\cap Y=l_{2,3}$,
see Proposition~\ref{prop:211}\ref{prop:211-4}.
\end{proof}

\begin{proof}[Proof of Proposition~\xref{prop:intersection-of-Ai}]
Assume there is a point $P\in Y\setminus l_{2,3}$. 
By Claims~\ref{claim:SS1} and~\ref{claim:SS2} \ $P\notin S_j$ for all $j$.
By Lemma~\ref{lemma:AS}\ref{lemma:AS-ii} for any $j=1,\ldots,4$ through $P$ passes a unique line 
$l_{j}\subset A_{S_j}$ meeting $S_j$. Since there are at most three lines on $V$ passing through $P$, one has $l_{i}=l_{j}$ for some $i\neq j$.
Set $l:=l_{i}=l_{j}$. Thus, one has $P\in l\subset A_{S_i}\cap A_{S_j}$. 
By Lemma~\ref{lemma:lines}\ref{lemma:lines:3},
$V$ contains no plane. 
Hence, the set of lines contained in the surface $A_{S_i}\cap A_{S_j}$ has dimension at most one, 
cf.~\cite[Lemma~A.1.1]{Kuznetsov-Prokhorov-Shramov}.
Since $l\not\subset S_j$ for all $j$, the line $l$ cannot be
$\Aut^0(V)$-invariant, see Proposition~\ref{prop:211}\ref{prop:211-4}.
Therefore, the stabilizer $G$ of $l$ under the $\Aut^0(V)$-action on $\Sigma(V)$ is one-dimensional. 
Since $l\not\subset S_j$ for all $j$, due to
Proposition~\ref{prop:211}\ref{cor:invariant-on-Va-4}-\ref{cor:invariant-on-Va-5}, $l$ is a ruling of $D$, and so, is $\Gm$-invariant. 
The rulings $l$ and $l_{2,3}$ of $D$ are disjoint, because $P\notin l_{2,3}$ by our choice. 
It follows from Claim ~\ref{claim:SS2} that $l\cap S_2=l\cap S_3=\varnothing$. So, we have $\{i,j\}=\{1,4\}$, that is, $l$ meets $S_1$ and $S_4$ in $\Gm$-fixed points. 
Besides, there exists a unique line $l_2\neq l$ in $V$ joining $P$ and $S_2$. 
If the stabilizer of $P$ in $G$ is finite, then $l=\overline{G\cdot P}\subset Y$. In particular, the intersection point $P_1$ of $l$ and $S_1$ lies on $Y$. 
By Claim ~\ref{claim:SS1} one has $P_1=v_2$, contrary to the fact that $l\cap S_2=\varnothing$. Hence, $P$ is fixed by $\Gm$,
and then the line $l_2$ is $\Gm$-invariant. Since $l_2$ is not a ruling of $D$ or of one of the $S_i$,
we get a contradiction with Proposition~\ref{prop:211}\ref{cor:invariant-on-Va-5}.
\end{proof}

\subsection*{Criteria of flexibility of affine cones}\label{ss:flex-crit}

To formulate the flexibility criteria, we need the following notions.

\begin{defi}[{\cite[Definitions 3--4]{Pe13}}]
\label{def:invariant} 
An open covering $(U_i)_{i\in I}$ 
of a projective variety $V$ by the $\A^1$-cylinders $U_i\cong\A^1\times Z_i$ is called 
\emph{transversal} if it does not admit any proper invariant subset. 
A subset $Y\subset V$ is proper if it is nonempty and different from $V$. 
It is called \emph{invariant} with respect to this covering if 
for any cylinder $ U_i\to Z_i$, $i\in I$, the intersection $Y\cap U_i$ is
covered by the fibers of 
$U_i\to Z_i$. 
\end{defi}

\begin{thm}
\label{thm:criterion-MPS} 
Let $(V,H)$ be a polarized smooth projective variety. 
Then the affine cone over $(V,H)$ is flexible if one of the following holds:
\begin{enumerate}
\item
\label{thm:criterion-MPS:1} 
\textup{(\cite[Theorem~5]{Pe13})} 
$V$ admits a transversal covering by a family of $\A^1$-cylinders 
$U_i=V\setminus \Supp(D_i)\cong \A^1\times Z_i$, where $Z_i$ is a smooth affine variety, $i\in I$;
\item
\label{thm:criterion-MPS:2} 
\textup{ (\cite[Theorem~1.4]{MPS18})} 
$V$ admits a
covering by a family of flexible Zariski open subsets $U_i=V\setminus \Supp(D_i)$, $i\in I$,
\end{enumerate}
where in the both cases the $D_i$ are effective $\QQ$-divisors on $V$ with $D_i\sim_{\QQ} H$, $i\in I$. 
\end{thm}

For instance, the affine cone over $(V,H)$ is flexible provided one can find an open covering 
$\{U_i\}_{i\in I}$ of $V$
by toric affine varieties $U_i=V\setminus \Supp(D_i)$ with no torus factor, that is, non-decomposable as a product 
$U_i=W_i\times (\A^1\setminus\{0\})$. Indeed, any such variety $U_i$
is flexible \cite[Theorem~2.1]{AKZ12}. In the simplest case where $\Pic(V)=\ZZ$ 
and $V$ admits an open covering by the affine spaces, the affine cone $X$ over $(V,H)$ is flexible 
whatever is an ample polarization $H$ of $V$. 
According to Theorem~\ref{thm:main}\ref{thm:main-a},\ref{thm:main-b} such a covering exists for any Fano-Mukai 
fourfold $V_{18}$ with a reductive automorphism group. In the remaining case of $V=V_{18}^{\aaa}$ 
with $\Aut^0(V)=\Ga\times\Gm$
we ignore whether $V$ admits a covering by flexible affine charts, 
cf.~Proposition~\ref{prop:intersection-of-Ai}. Hence, we cannot apply the criterion of 
Theorem~\ref{thm:criterion-MPS}\ref{thm:criterion-MPS:2} in this case. Instead, 
we will apply the following version, which mixes the two criteria of Theorem~\ref{thm:criterion-MPS}. 

Let $(V,H)$ be a polarized projective variety, and let $A$ be
an effective divisor on $V$. 
We say that an open set $U=V\setminus \Supp(A)$ is \emph{polar} if $A\sim_{\QQ} H$. 

\begin{prop}
\label{thm:criterion-new}
Let $(V,H)$ be a smooth projective variety of dimension $n\ge 3$ with an ample polarization. 
Suppose $V$ possesses a family of flexible polar $\A^1$-cylinders 
$U_i=V\setminus \Supp(A_i)\cong\A^1\times Z_i$, $i\in I$,
where $Z_i$ is an affine variety of dimension $n-1$. 
Assume further that 
\begin{itemize}
\item
$D:=\bigcap_{i\in I} A_i$ is a subvariety of $V$ of dimension $m\le n-2$;
\item
through any point of $V$ pass at most $k$ components of $D$;
\item
any point $P\in D$ is contained either in a principal open polar $\A^{m+1}$-cylinder $U_P\cong\A^{m+1}\times Z_P$ 
in $V$, or in $k+1$ principal polar open $\A^m$-cylinders $U_{P,j}\cong \A^m\times Z_{P,j}$ in $V$, $j=1,\ldots,k+1$, 
where $Z_P$ and $Z_{P,j}$ are affine varieties, and for any two cylinders $U_{P,j}$ and $U_{P,j'}$ with $j\neq j'$ 
the $\A^m$-fibers through $P$ of the natural projections $U_{P,j}\to Z_{P,j}$ and $U_{P,j'}\to Z_{P,j'}$ 
meet properly, that is, the dimension of their intersection is smaller than $m$.
\end{itemize}
Then the affine cone over $(V,H)$ is flexible.
\end{prop}

\begin{proof}
We use the criterion of Theorem~\ref{thm:criterion-MPS}\ref{thm:criterion-MPS:1}, 
that is, we show the existence of a transversal covering of $V$ by polar open $\A^1$-cylinders. 
For such a covering, we take the union of the collections $\{U_i\}$, $\{U_P\}$, and $\{U_{P,j}\}$, 
where each member is endowed with all possible structures of an $\A^1$-cylinder. 

Let $Y\subset V$ be a nonempty subset invariant with respect to the above covering of $V$ by $\A^1$-cylinders, 
see Definition~\ref{def:invariant}. We claim that
if $Y\cap U_i\neq\varnothing$ for some $i\in I$, then $Y\supset U_i$. 
Indeed, let $P\in Y\cap U_i$, and let $l$ be the ruling of the $\A^1$-cylinder $U_i$ passing through $P$. 
Then $l\subset Y$ because $Y$ is invariant. 
Since $U_i$ is flexible, for any point $P'\in U_i$ different from $P$ one can find an automorphism $\alpha\in\SAut(U_i)$ 
such that $\alpha(P)=P$ and $\alpha^{-1}(P')\in l$. Then $\alpha(l)$ is a ruling of a new $\A^1$-cylinder structure on $U_i$. 
Since $P,P'\in \alpha(l)$ and $P\in Y$, where $Y$ is invariant, then also $P'\in \alpha(l)\subset Y$. Hence, one has $U_i\subset Y$, as claimed.

It follows that
$Y\cap U_j\neq\varnothing$ for any $j\in I$. Thereby, one has $Y\supset \bigcup_{i\in I} U_i=V\setminus D$. 
Due to our assumptions, for any point $P\in D$ one can choose either a principal open $\A^{m+1}$-cylinder $U_{P}$, 
or a principal open $\A^m$-cylinder $U_{P,j}$ such that the $\A^m$-fiber passing through $P$ of the projection $U_{P,j}\to Z_{P,j}$ 
is not contained in $D$. Then for a suitable $\A^1$-cylinder structure on $U_P$ or $U_{P,j}$, respectively, the ruling $l\cong\A^1$ 
passing through $P$ is not contained in $D$, and so, meets $V\setminus D\subset Y$. Then one has $P\in l\subset Y$. 
Since this holds for any $P\in D$, one has $D\subset Y$, and so, $Y=V$, that is, $Y$ is not a proper subset of $V$. 

The latter argument shows as well that a nonempty invariant subset $Y\subset V$ cannot be contained in $D$. 
Thus, the criterion of Theorem~\ref{thm:criterion-MPS}\ref{thm:criterion-MPS:1} applies and yields the result. 
\end{proof}

Using Proposition~\ref{prop:intersection-of-Ai} we deduce such a corollary.

\begin{cor} 
\label{cor:criterion-V}
The affine cones over the Fano-Mukai fourfold $V$ with $\Aut^0(V)=\Ga\times\Gm$ are flexible provided any point $P$ 
of the common ruling $l_{2,3}$ of the cubic cones $S_2$ and $S_3$ on $V$ is contained in a principal open $\A^2$-cylinder $U_{P}$ in $V$.
\end{cor}

In Proposition~\ref{prop:A2-cylinders} we construct such an open covering of $l_{2,3}$ in $V$ by principal open $\A^2$-cylinders. 
Combining with Corollary~\ref{cor:criterion-V} this gives a proof of Theorem~\ref{mthm} in the remaining case $\Aut^0(V)=\Ga\times\Gm$.

\section{$\A^2$-cylinders in smooth quadric fourfolds and in the del Pezzo quintic fourfold} 
\label{section:quadric}

The next lemma on the existence of an $\A^2$-cylinder will be used the proof of Proposition~\ref{prop:cylinder0} below. 

\begin{lem}
\label{lem:link0}
Let $Q\subset \PP^5$ be a smooth quadric, and let $Q',\, Q^{\star}$ be distinct hyperplane sections of $Q$.
Let $Q_1,\dots,Q_k$ be the members of the pencil $\langle Q',\, Q^{\star}\rangle$ generated by $Q'$ and $Q^{\star}$
which have singularities outside $Q'\cap Q^{\star}$, and let $P_i$ be the unique singular point of $Q_i$. 
Given a point $P\in Q\setminus (Q'\cup Q^{\star}\cup \{P_1,\dots,P_k\})$ there exists a principal affine 
open subset $U=U_P\subset Q\setminus (Q'\cup Q^{\star})$
such that
\begin{enumerate}
\item
\label{lem:link1}
$P\in U$;
\item
\label{lem:link2}
$U\cong \A^2\times Z$, where $Z$ is an affine surface.
\end{enumerate}
\end{lem}

\begin{proof}
Pick a general point $P^{\bullet}\in Q'\cap Q^{\star}$, and let $\mathbf{T}_{P^{\bullet}}Q\subset \PP^5$ be the embedded tangent space to $Q$ at $P^{\bullet}$. 
The projection with center $P^{\bullet}$ defines an isomorphism 
\[
Q\setminus Q^{\bullet}\cong \PP^4\setminus \PP^3 \cong \A^4,\quad\text{where}\quad Q^{\bullet}:=Q\cap \mathbf{T}_{P^{\bullet}}Q.
\] 
The quadric cone $Q^{\bullet}$ with vertex $P^{\bullet}$ coincides with the union of lines on $Q$ passing through $P^{\bullet}$. 
If the quadric cone $\Delta_P(Q)=Q\cap \mathbf{T}_PQ$ with vertex $P\in Q\setminus (Q'\cap Q^{\star})$ contains $Q'\cap Q^{\star}$, then $\Delta_P(Q)$ 
coincides with a member $Q_i$ of the pencil $\langle Q', Q^{\star}\rangle$, which has the singular point $P=P_i$ for some $i\in\{1,\ldots,k\}$. 
However, the latter is excluded by our assumption. So, $Q'\cap Q^{\star}\not\subset \Delta_P(Q)$. Hence, for
the general point $P^{\bullet}\in Q'\cap Q^{\star}$ the line joining $P$ and $P^{\bullet}$ is not contained in $Q$. The latter implies $P\notin Q^{\bullet}$.

The images of $Q'\setminus Q^{\bullet}$ and $Q^{\star}\setminus Q^{\bullet}$ in $\A^4=\PP^4\setminus \PP^3$ under the projection with center $P^{\bullet}$ 
is a pair of affine hyperplanes with nonempty intersection.
Thus, we obtain
\begin{equation}
\label{eq:projection}
P\in Q\setminus (Q^{\bullet}\cup Q'\cup Q^{\star})\cong \A^2\times (\A^1\setminus \{\text{a point}\})\times (\A^1\setminus \{\text{a point}\}).\qedhere
\end{equation} 
\end{proof}

Recall that a smooth del Pezzo quintic fourfold $W=W_5\subset \PP^7$ is unique up to isomorphism \cite{Fujita1981}. 
This variety is quasihomogeneous, more precisely, the automorphism group $\Aut(W)$ 
has the open orbit $W\setminus R\cong\A^4$ in $W$, where $R$ is the hyperplane section of $W$ covered by the lines on $W$ 
which meet the unique $\sigma_{2,2}$-plane $\Xi\subset W$, see \cite[Section~4]{PZ16}. The planes on $W$ different from $\Xi$ 
form a one-parameter family, and their union coincides with $R$; we call them \emph{$\Pi$-planes}. 

\begin{prop}
\label{prop:link1}
Let $W=W_5\subset \PP^7$ be the del Pezzo quintic fourfold. Then the following hold.
\begin{enumerate} 
\item
\label{prop:link1-0} 
\textup{(\cite[Corollary~2.6]{PZ16})}
The Hilbert scheme $\Sigma(W)$ of lines on $W$ is smooth, irreducible of dimension $\dim \Sigma(W)=4$.
For any point $P\in W$ the Hilbert scheme $\Sigma(W;P)\subset \Sigma(W)$ of lines passing through $P$ has pure dimension $1$.

\item
\label{prop:link1-a} 
\textup{(\cite[Proposition~4.11.iv]{PZ16})}
For any line $l\subset W$ there exists a unique hyperplane section $B_l$ of $W$ with $\Sing(B_l)\supset l$. 
This $B_l$ is the union of lines meeting $l$.

\item
\label{prop:link1-new} 
Given a point $P\in W$, let $\Delta_P$ be the union of lines in $W$ passing through $P$.
If $P\in W\setminus R$, then $\Delta_P$ is a cubic cone. 
If $P\in R\setminus \Xi$, then $\Delta_P$ is the union of a plane $\Pi_{P}$ passing through $P$ 
and a quadric cone $\Delta_P'$ with vertex $P$.

\item
\label{prop:link1-a2a} 
Let $B\subset W$ be a hyperplane section whose singular locus is two-dimensional.
Then $B=R$ and $\Sing(B)=\Xi$.

\item
\label{prop:link1-aa} 
Let $B\subset W$ be a hyperplane section whose singular locus is one-dimensional.
Then $B=B_l$ for some line $l$.

\item
\label{prop:link1-aaa} 
Let $B\subset W$ be a hyperplane section, and let $C\subset B$ be an irreducible curve. 
Assume $B$ contains a two-dimensional family of lines meeting $C$. 
Then one of the following holds:
\begin{enumerate}
\item
$C$ is contained in a plane on $B$;
\item
$C=l$ is a line, and $B=B_l$;
\item
$B=R$. 
\end{enumerate}
\end{enumerate}
\end{prop}

\begin{proof}
\ref{prop:link1-new} 
By~\ref{prop:link1-0} the universal family of lines $\fL(W)\subset \Sigma(W)\times W$ 
is smooth, and the natural projection $s:\fL(W)\to W$ 
is a flat morphism of relative dimension one. Its fiber $s^{-1}(P)$ is isomorphic to
the base of the cone $\Delta_P$.
Let $P\in W\setminus R$.
Since $W\setminus R$ is the open orbit of $\Aut(W)$, see, e.g., \cite[(5.5.5)]{PZ18}, 
the fiber $s^{-1}(P)$ is smooth in this case. 
Let $H$ be a general hyperplane section of $W$ passing through $P$.
By Bertini's theorem, $H$ is smooth, and by the adjunction formula, $H$ is a del Pezzo 
threefold of degree $5$. It is well known, see, e.g., \cite[Ch.~2, \S~1.6]{Iskovskikh-1980-Anticanonical} 
or \cite[Corollary~5.1.5]{Kuznetsov-Prokhorov-Shramov},
that through a general point of $H$ pass exactly three lines.
This implies $\deg \Delta_P=3$.
On the other hand, one has $\Delta_P=\mathbf{T}_{P}W\cap W$. Since $W$ is intersection of 
quadrics, $\Delta_P$ cannot be a cone over a plane cubic. 
It follows that $\Delta_P$ is a cubic cone. 

Let further $P\in R\setminus \Xi$. 
By \cite[Theorem~6.9]{Piontkowski1999}, $\Aut(W)$ acts transitively on $R\setminus \Xi$. 
We have $\Delta_P\not \supset \Xi$, and $\Delta_P$ contains a $\Pi$-plane
$\Pi_P$ passing through $P$. Such a plane is unique because no two planes on $W$ meet 
outside $\Xi$. By the flatness of $s$ we have $\deg \Delta_P=3$. 
There are lines on $W$ which are not contained in $R$ and meet $R\setminus \Xi$, 
and one of these lines passes through $P$. Therefore, one has
$\Delta_P\neq 3\Pi_P$.
Then the only possibility is $\Delta_P=\Pi_P+\Delta'_P$, where 
$\Delta_P'$ is a quadric cone.

\ref{prop:link1-a2a}
Let $Z_2$ be an irreducible component of $ \Sing(B)$ of dimension two.
Choose general hyperplane sections $H_1$ and $H_2$ of $W$, and let $C=B\cap H_1\cap H_2$.
By Bertini's theorem, $C$ is an irreducible curve with $\Sing(C)=Z_{2}\cap H_1\cap H_2$. 
By the adjunction formula one has $\mathrm{p}_a(C)=1$.
Hence $\deg Z_{2}=1$, i.e. $Z_2$ is a plane. Since $B$ contains any line meeting $Z_2$
we have $B=R$, and then $\Sing(B)=\Xi$.

\ref{prop:link1-aa}
Let $Z_1$ be the union of one-dimensional irreducible components of $\Sing(B)$.
Since $Z_1\subset \Sing(B)$ and $B$ is a hyperplane section, $B$ contains any line meeting $Z_1$.
Hence, $B$ is the union of lines meeting $Z_1$, see~\ref{prop:link1-0}. By~\ref{prop:link1-a2a} one has $B\neq R$.
If $Z_1$ is a line, then $B=B_{Z_1}$ by~\ref{prop:link1-a}.
Thus, we may assume $Z_1$ to be a curve of degree $d>1$.
Consider the general hyperplane section $H\subset B$. By Bertini's theorem,
$H$ has exactly $d$ singular points, and these are the points of $H\cap Z_1$.
By the adjunction formula, $-K_H$ is the class of a hyperplane section of $H$. Hence, 
$H$ is a normal Gorenstein del Pezzo quintic surface.

By \cite[Proposition~8.5]{Coray-Tsfasman-1988} we have $d=2$, that is, $H$ 
has exactly two singular points, say,
$P_1$ and $P_2$, and $Z_1$ is a conic. 
Again by \cite[Proposition~8.5]{Coray-Tsfasman-1988},
$H$ contains the line joining $P_1$ and $P_2$. Hence, $B$ contains the linear span $\langle Z_1\rangle$. 
Since $R$ is the union of planes contained in $W$, we have $Z_1\subset \langle Z_1\rangle\subset R$. 
By~\ref{prop:link1-a2a} one has $\langle Z_1\rangle\neq \Xi$. So,
$\langle Z_1\rangle=\Pi$ is a $\Pi$-plane on $R$, where $\Pi\subset B$. 
Take a general line $l\subset \Pi$, and let $Z_1\cap l=\{P_1,\, P_2\}$.
Since $\Xi$ meets $l\subset \Sing(B_l)$ one has $\Xi\subset B_l$. 
Likewise, since $\Xi$ meets $Z_1\subset \Sing(B)$ one has $\Xi\subset B$. 
Besides, the quadric cones $\Delta_{P_1}'$ and $\Delta_{P_2}'$ as in~\ref{prop:link1-new}
are contained in both $B$ and $B_l$. Thus, we obtain 
\[
B\cap B_l \supset \Pi\cup \Xi\cup \Delta_{P_1}'\cup \Delta_{P_2}'.
\]
Assuming $B\neq B_l$ the latter contradicts the fact that $\deg (B\cap B_l)=\deg W=5$.

\ref{prop:link1-aaa} 
We may suppose that $B\neq R$ and the singular locus of $B$ 
has dimension $\le 1$, see~\ref{prop:link1-a2a}. 
By assumption, $C\subset B$ is an irreducible curve,
and there exists an irreducible two-dimensional family of lines $\Sigma(B,C)\subset \Sigma(B)$ on $B$ meeting $C$. 
Let $r: \fL(C, B)\to \Sigma(C, B)$ be the universal family, 
and let $s: \fL(C, B)\to B$ be the natural projection.
If $s(\fL(C, B))\neq B$, then $s(\fL(C, B))$ is a plane, see, e.g., 
\cite[Lemma~A.1.1]{Kuznetsov-Prokhorov-Shramov}, and so, $C$ is contained in a plane on $B$. 
Assume further $s(\fL(C, B))= B$, and so, $s$ is a generically \'etale morphism. 
We claim that the general line from $\Sigma(C, B)$ meets the singular locus 
of $B$. 
The argument below is well-known, see e.g. 
\cite[Ch. 3, Prop.~1.3]{Iskovskikh-1980-Anticanonical} or
\cite[Lemma~2.2.6]{Kuznetsov-Prokhorov-Shramov}, 
and we repeat it in brief for the sake of completeness. 
Suppose to the contrary that the general line $l$ from $\Sigma(C, B)$
lies in the smooth locus of $B$. Using the fact that the restriction of $s$ to 
$r^{-1}([l])$ is an isomorphism, we may identify $l$ with $r^{-1}([l])$.
For the normal bundles of $l$ we have
\begin{equation}
\label{eq:norm-bundle}
\NNN_{l/\fL(C, B)} =\OOO_{l}\oplus \OOO_{l},
\qquad 
\NNN_{l/B} =\OOO_{l}(a)\oplus \OOO_{l}(-a),\quad a\ge 0.
\end{equation} 
Over the point $l\cap C$ the map $s$ is not an isomorphism.
Hence the differential 
\[
d s: \NNN_{l/\fL(C, B)} \xlongrightarrow{\qquad}\NNN_{l/B}
\]
is not an isomorphism either. From \eqref{eq:norm-bundle} we see that 
$ds$ degenerates along $l$. This means that $s$ is not generically \'etale,
a contradiction.

Thus, the general line from $\Sigma(C, B)$ meets $\Sing(B)$.
Since $B$ is not a cone, by our assumption we have $\dim \Sing(B)=1$.
Due to~\ref{prop:link1-aa} there is a line $l_0$ on $W$ such that $B=B_{l_0}$.
If $l_0=C$, then we are done. Otherwise, the lines on $B$ passing through 
the general point $P\in C$ meet $l_0=\Sing(B)$. The union of these lines is 
the cone with vertex $P$ over $l_0$, that is, a plane. It follows that $B$ is swept out by planes. 
Since any plane on $W$ is contained in $R$,
we conclude that $B=R$, contrary to our assumption.
\end{proof}

\begin{prop}
\label{prop:cylinder0}
Let $B\subset W$ be a hyperplane section.
For any point $P\in W\setminus B$ there exists a principal affine open subset
$U_{P}\subset W\setminus B$ such that
\begin{enumerate}
\item
\label{prop:cylinder1}
$P\in U_P$;
\item
\label{prop:cylinder2}
$U_P\cong \A^2\times Z_P$, where $Z_P$ is an affine surface.
\end{enumerate}
\end{prop}
\begin{proof}
If $B=R$ then $W\setminus B\cong\A^4$ \cite[Corollary~2.2.2]{PZ18}, and the assertion follows.
So, we assume in the sequel $B\neq R$. 
We apply the following construction from \cite{Fujita1981} (see also \cite[Proposition~4.11]{PZ16}). 
Fix a line $l\subset W$ not contained in $R$.
There is the Sarkisov link
\begin{equation}
\label{diag:link1}
\vcenter{
\xymatrix@R=1em{
&\tilde W\ar[dl]_{\rho_W}\ar[dr]^{\varphi_Q}&
\\
W\ar@{-->}[rr]^{\theta_l}&& Q
} }
\end{equation}
where $Q\subset\PP^5$ is a smooth quadric, 
$\theta_l$ is induced by 
the linear projection with center $l$, $\rho_W$ is the blowup of $l$, and
$\varphi_Q$ is the blowup of a smooth cubic scroll $\Lambda=\Lambda_3\subset Q$.
Furthermore, $\varphi_Q$ sends the $\rho_W$-exceptional divisor $E_W\subset\tilde W$ 
onto the quadric cone $Q^{\star}=Q\cap \langle\Lambda\rangle$, while the $\varphi_Q$-exceptional divisor $\tilde B_l\subset \tilde W$
is the proper transform of $B_l\subset W$. 
Let $Q':=\theta_l(B)$. 

Suppose our line $l\subset W$ satisfies the following conditions:
\begin{enumerate}
\renewcommand\labelenumi{\rm (\Alph{enumi})}
\renewcommand\theenumi{\rm (\Alph{enumi})}
\item
\label{cond-1}
$l\subset B$, $l\not \subset R$, $P\notin B_l$, and 
\item
\label{cond-2} 
$\theta_l(P)\notin \{P_1,\ldots,P_k\}$, where $P_1,\ldots,P_k$ have the same meaning as in Lemma~\ref{lem:link0}.
\end{enumerate}
We have an isomorphism
\[
W\setminus (B\cup B_l)\cong Q\setminus (Q'\cup Q^{\star}).
\] 
By Lemma~\ref{lem:link0} there exists a principal affine open subset $\tilde U_P\subset Q\setminus (Q'\cup Q^{\star})$
such that
\begin{enumerate}
\item
\label{prop:cylinder-1a}
$\theta(P)\in \tilde U_P$;
\item
\label{prop:cylinder-2a}
$\tilde U_P\cong \A^2\times Z_P$, where $Z_P$ is an affine surface. 
\end{enumerate}
Then $U_P=\theta^{-1}(\tilde U_P)\subset W$ verifies~\ref{prop:cylinder1}--\ref{prop:cylinder2} of Proposition~\ref{prop:cylinder0}. 

It remain to show the existence of a line $l\subset W$ satisfying~\ref{cond-1} and~\ref{cond-2}.

Consider the union $\Delta_P$ of lines on $W$ passing through $P$.
By Proposition~\ref{prop:link1}\ref{prop:link1-new}, $\Delta_P$ is a (possibly reducible) cubic cone 
with vertex $P$ on $W$ of pure dimension two. Recall that
$\Delta_P=(\mathbf{T}_{P}W\cap W)_{\red}$, because $W$ is an intersection of quadrics. 

Let $C=\Delta_P\cap B$, and let $\Sigma(B,C)$ be the Hilbert scheme of lines in $B$ which meet $C$, 
or, which is equivalent, which meet $\Delta_P$. If $\dim\,\Sigma(B,C)=2$ then
by Proposition~\ref{prop:link1}~\ref{prop:link1-aaa}, either
the lines from $\Sigma(B,C)$ sweep out a plane, say, $D$ on $B$, or $C=l$ is a line and $B=B_l$. 
However, the latter case is impossible. Indeed, $B_l$ being singular along $l=C$, any ruling of the 
cone $\Delta_P$ meets $C\subset\Sing(B_l)$, hence is contained in $B_l=B$. Then also $P\in B$, which is a contradiction. 

Assume further $\dim\, \Sigma(B,C)=1$. Then the lines from $\Sigma(B,C)$ sweep out a surface scroll $D$ on 
$B$.
Thus, the latter holds whatever is the dimension of $\Sigma(B,C)$.

It follows from Proposition~\ref{prop:link1}\ref{prop:link1-0} that the threefold 
$B$ is covered by lines on $W$. Since $B\neq R$ by our assumption, there is a point $P'\in B\setminus (D\cup R)$. 
Any line $l$ through $P'$ on $B$ does not lie on $D\cup R$, and so, does not belong to $\Sigma(B,C)$.
Hence, $\Delta_P\cap l=\varnothing$ for the general line $l$ on $B$. Since $B_l$ is the union of lines on $W$ meeting $l$, 
see Proposition~\ref{prop:link1}~\ref{prop:link1-a}, we deduce $P\notin B_l$. Thus, the general line $l$ on $B$ satisfies~\ref{cond-1}. 

To show~\ref{cond-2}, we use the notation from the proof of (A). By the preceding, one has $\Delta_P\cap l=\varnothing$. 
Hence, the projection $\theta_l$ with center $l$ is regular in a neighborhood of $\Delta_P$. It follows that
$\Delta_P^Q:=\theta_l(\Delta_P)$ is a cone with vertex $P_Q:=\theta_l(P)$. 
Suppose to the contrary that $P_Q=P_i$ for some $i\in\{1,\ldots,k\}$. Then $P_Q$ is the vertex of a 
quadric cone $Q_i$ over $Q'\cap Q^{\star}$, see Lemma~\ref{lem:link0}. Clearly, one has $\Delta_P^Q\subset Q_i$. 
From $l\not\subset D$ we deduce $C\not\subset B_l$. 

On the other hand, we have $\theta_l(B_l)=\Lambda$,\, $Q^{\star}=Q\cap \langle\Lambda\rangle$, and 
\[\theta_l(C)=\theta_l(\Delta_P\cap B)\subset\theta_l(\Delta_P)\cap \theta_l(B)=\Delta_P^Q\cap Q'\subset Q_i\cap Q'=Q^*\cap Q'\subset\langle\Lambda\rangle\cap Q.\]
From these inclusions we deduce $B_l=\theta_l^{-1}(\Lambda)\subset\theta_l^{-1}(\langle\Lambda\rangle\cap Q)$. 
The latter inclusion is an equality, since the both sets are hyperplane section of $W$. Thus, one has $C\subset B_l$. This contradiction proves~\ref{cond-2}.
\end{proof}

\section{$\A^2$-cylinders in $V_{18}$ and flexibility of affine cones over $V^{\aaa}_{18}$}
\label{sec:V18-aff-cones}
In this section we finish the proofs of Theorems~\ref{mthm} and~\ref{thm:main} in the case where $\Aut^0(V)=\Ga\times\Gm$. 
The assertion of Theorem~\ref{mthm} follows immediately in this case by the flexibility criterion of 
Corollary~\ref{cor:criterion-V} due to the following result.

\begin{sit}\label{sit:Sarkisov-link}
 Let $S$ be a cubic scroll on $V=V_{18}$. Then either $S$ is smooth and isomorphic to the Hirzebruch surface $\FF_1$, or $S$ is a cubic cone. By virtue of \cite[Proposition~3.1]{PZ18}, in both cases the linear projection 
\[
\theta_S\colon V=V_{18}\subset \PP^{12}\xdashrightarrow{\hspace{2em}}\PP^7
\]
with center $\langle S\rangle=\PP^4$ restricted to $V$ yields a Sarkisov link 
\begin{equation}
\label{eq:link} 
\vcenter{
\xymatrix@R=1em{
&\tilde V\ar[dl]_{\rho_V}\ar[dr]^{\varphi_W}&
\\
V\ar@{-->}[rr]^{\theta_S}&& W
} }
\end{equation}
where $W=W_5\subset \PP^7$ is the del Pezzo quintic fourfold, $\rho_V$ is the blowup of $S$, and
$\varphi_W$ is the blowup of a smooth rational quintic scroll $F=F_5\subset W$ isomorphic to $\FF_1$.
Furthermore, $\varphi_W$ sends the $\rho_V$-exceptional divisor $E_V\subset\tilde V$ 
onto the hyperplane section $B_F:=W\cap \langle F\rangle$ of $W$, while the 
$\varphi_W$-exceptional divisor $\tilde A_S\subset \tilde V$
is the proper transform of $A_S$. We have
\[
V\setminus A_S\cong W\setminus B_F.
\]
If $S$ is a cubic cone, then $B_F=R$ is the hyperplane section of $W$ singular along the $\Xi$-plane, 
and $V\setminus A_S\cong W\setminus R\cong\A^4$. 
 \end{sit}

\begin{prop}
\label{prop:A2-cylinders} 
Let $V=V_{18}$ be a Fano-Mukai fourfold of genus $10$. 
Then for any point $P\in V$ there exists a principal affine open subset $U_{P}\subset V$
such that
\begin{enumerate}
\item
\label{prop:cylinder-1}
$P\in U_P$;
\item
\label{prop:cylinder-2}
$U_P\cong \A^2\times Z_P$, where $Z_P$ is an affine surface. 
\end{enumerate}
\end{prop}

\begin{proof}
Let $S\subset V$ be a cubic scroll, and let $A_S$ be the hyperplane section of $V$ 
with $\Sing(A_S)=S$, see Lemma~\ref{lemma:AS}. The general cubic scroll $S\in\SSS(V)$ 
is smooth, that is, not a cubic cone, and $P\notin A_S$ by Lemma~\ref{lemma:capAS}.
By Lemma~\ref{lemma:capAS}, for the general $S\in\SSS(V)$ the map $\theta_S$ in 
\eqref{eq:link} is well defined at $P$, and $\theta_S(P)\notin B_F$.
Now the assertion follows from Proposition~\ref{prop:cylinder0}.
\end{proof}

\newcommand{\etalchar}[1]{$^{#1}$}
\def\cprime{$'$}

\end{document}